\documentclass[11pt]{article}

\usepackage{amsmath}
\usepackage{amsthm}
\usepackage{amssymb}
\usepackage{mathtools}
\usepackage{graphicx}
\usepackage{float}
\usepackage{cite}
\usepackage{titling}

\newcommand{\RR}{\mathbb{R}}

\newtheorem{defn}{Definition}[section]
\newtheorem{thm}[defn]{Theorem}

\newcommand{\is}{\simeq}

\title{On the Number of Edges in Maximally Linkless Graphs}
\author{Max Aires \\ 
Dept.\ Math.\ Sciences \\ Carnegie Mellon University \\ Pittsburgh, PA\\ 
maires@andrew.cmu.edu}

\begin{document}

\maketitle

\begin{abstract}
\noindent
A maximally linkless graph is a graph that can be embedded in $\RR^3$ without any links, but cannot be embedded in such a way if any other edge is added to the graph.
Recently, a family of maximally linkless graphs was found with $m=3n-3$ edges. We improve upon this by demonstrating a new family of maximally linkless graphs with $m\le \frac{14}{5}n$ edges. 
\end{abstract}


\section{Introduction}

A graph can be \begin{it}embedded\end{it} in $\RR^3$ by taking points as vertices and curves between those points as edges, where no edge intersects another edge or vertex except at its endpoints. Two cycles in a graph embedding are \begin{it}{linked}\end{it} if they cannot be continuously moved to create two disjoint cycles on a common plane without their edges crossing. A graph which can be embedded in $\RR^3$ with no linked cycles is called \begin{it}linkless\end{it}. We remark that in general a link in an embedded graph could involve more than two cycles, but the minor characterization of Robertson, Seymour, and Thomas~\cite{RST} shows that it suffices to consider pairs of disjoint cycles. 

A graph embedding is \begin{it}flat\end{it} if for each cycle there exists a disk with the cycle as its boundary that is otherwise disjoint from the embedding. Robertson, Seymour, and Thomas showed that a graph has a flat embedding if and only if it has a linkless embedding~\cite{RST}.

In 1983, Sachs showed that $K^6$ and six other graphs are not linkless~\cite{Sachs}. Moreover, he observed that the family of linkless graphs is closed under the operations of taking minors and $Y-\Delta$ transforms. Linkless graphs share some analogous properties with planar graphs. Robertson, Seymour, and Thomas extended Sachs' result by showing that linkless graphs can be characterized precisely as any graph which does not contain any of seven forbidden graph minors~\cite{RST}. This result is analogous to the result that planar graphs are characterized by the forbidden minors of $K^5$ and $K_{3,3}$.

The similarities between planar and linkless graphs suggest trying to generalize other properties, particularly the structure of maximally planar graphs. We say a planar graph is \begin{it}maximally\end{it} planar if adding any edge makes the graph nonplanar. Any planar graph on $n$ vertices has a maximum of $3n-6$ edges. Moreover, any maximally planar is a triangulation, and has the full $3n-6$ edges. By a classical result of Mader, any graph on $n$ vertices with no $K^6$ minor, and hence any linkless graph, has at most $4n-10$ edges~\cite{Mader}. Some linkless graphs do in fact have this many edges, such as \begin{it}apex\end{it} graphs, which consist of a maximally planar graph with an extra vertex connected to every vertex in the planar graph. 

A class of maximally linkless graphs was recently found by Dehkordi and Farr on $n$ vertices with $m= 3n-3$ edges~\cite{Dehkordi/Farr}. They asked whether this was a lower bound; we shall show it is not. We will present a graph on $13$ vertices with $31$ edges, and will show that this gives a way of generating maximally linkless graphs with less than $\frac{14}{5}n$ edges for arbitrarily large $n$. This result highlights a key difference between planar and linkless graphs: while maximally planar graphs have a very simple structure and predictable number of edges, maximally linkless graphs have no such structure, and the number of edges can vary significantly. In particular, a maximally linkless graph on $n$ vertices can have less vertices than a maximally planar graph on $n$ vertices, a surprising result.

\section{Main Results}

\begin{thm} The graph $G$ below is maximally linkless.
\end{thm}

\begin{figure}[H]
    \center{\includegraphics[width=360pt]{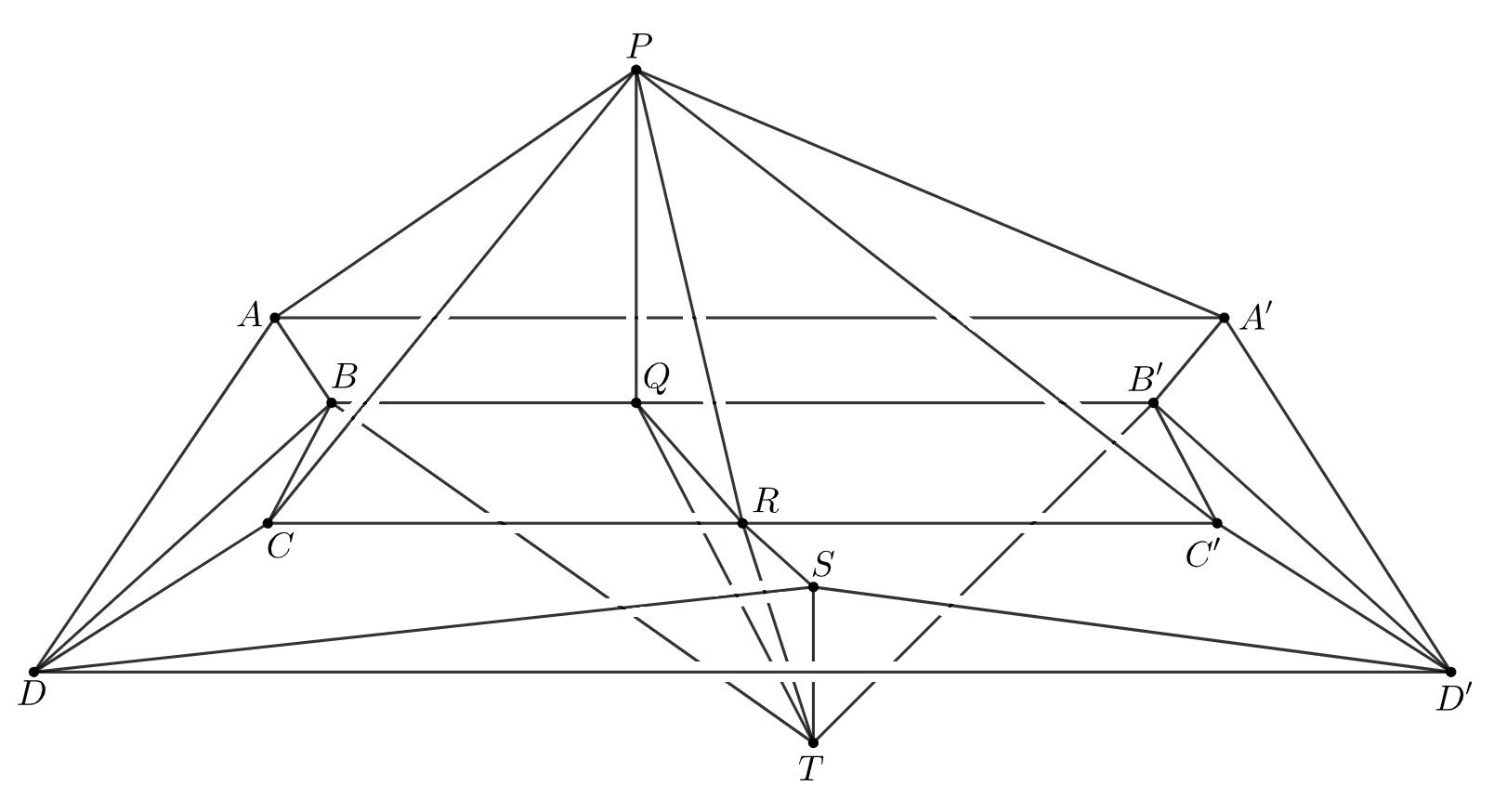}}
    \caption{The graph $G$}
    \label{fig:graph}
\end{figure}

\begin{proof} In the embedding above, $G\backslash\{P,T\}$ is embedded on a horizontal plane with $P$ above the plane and $T$ below. Suppose this embedding were not linkless. Then (since the standard embedding of apex graphs are linkless) $G$ must contain a link passing through both $P$ and~$T$. It is impossible for two cycles to be linked with one containing $P$ and one containing~$T$, for then one would lie (non-strictly) above the plane and the other below. So if the embedding is linked, one cycle must contain both $P$ and~$T$. 

Hence the planar component $G\backslash\{P,T\}$ would have to contain a cycle, and the other cycle would have to pass through this plane at two points, one within the cycle and one without. Such a cycle would have to separate two points from among $T$'s neighbors, or $B,Q,R,S,B'$. Clearly there are no cycles separating any pair from $B,Q,B',R$, so it must separate $S$ from one of $B,Q,B'$, so the planar cycle must include $D$ and $D'$. But then the nonplanar cycle would have to include $R$, meaning the planar cycle would separate $R$ from $B,Q,$ or $B'$, which is impossible. Hence no link exists and our embedding of $G$ is linkless.

We shall now show that $G$ is maximal. In particular, we shall show that adding any edge creates a $K^6$ minor. Observe that the graph contains an automorphism by swapping $A, B, C, D$ with $A', B', C', D'$ respectively and fixing $P,Q,R,S,T$. In this way, we can use symmetry to limit the number of edges we must consider from 47 to 26. If we partition the 13 vertices into six sets where the vertices within each set are connected and every pair of sets are connected by an edge except for one pair, then we can conclude that no edge can be added between those two sets. Consider eight different ways of partitioning the vertices of $G$ below; in each every pair of vertex sets are connected except for the first two:  

\begin{center}\begin{tabular}{|c|cccccc|c|}
\hline
Case & & & Partition & & & & Eliminated Edges \\
\hline
1 & $\bf{ABCD}$ & $\bf{B'C'}$ & $STD'$ & $PA'$ & $Q$ & $R$ & $AB',AC',BB',BC',CC',BD',CD'$ \\\hline
2 & $\bf{DS}$ & $\bf{PQ}$ & $ABA'$ & $TB'$ & $CR$ & $C'D'$ & $DP,DQ,SP,SQ$ \\\hline
3 & $\bf{CRST}$ & $\bf{A}$ & $DC'D'$ & $BQ$ & $A'B'$ & $P$ & $AC,AR,AS,AT$ \\\hline
4 & $\bf{BST}$ & $\bf{PC'}$ & $QRB'$ & $AA'$ & $CD$ & $D'$ & $BP,PS,PT,BC',CS,CT$\\\hline
5 & $\bf{AD}$ & $\bf{QR}$ & $BST$ & $CP$ & $A'B'$ & $C'D'$ & $AQ,AR,DQ,DR$\\\hline
6 & $\bf{AB}$ & $\bf{RC'D'}$ & $DST$ & $CP$ & $A'B'$ & $Q$ & $AR,BR,AC',AD',BC',BD'$\\\hline
7 & $\bf{CDS}$ & $\bf{QB'}$ & $ABA'$ & $RT$ & $C'D'$ & $P$ & $CQ,BC',DQ,QS,BD',BS$\\\hline
8 & $\bf{TB'C'}$ & $\bf{CD}$ & $AA'D'$ & $BQ$ & $RS$ & $P$ & $BC',BD',CC',CD',CT,DT$\\\hline
\end{tabular}\end{center}

Below is a list of the 26 edges we must consider, and which case shows that adding them creates a $K^6$ minor.

\begin{center}\begin{tabular}{|c|c||c|c|}
\hline
Edge            &Case(s)&   Edge            &Case(s)\\
\hline
$AC(\is A'C')$  &3      &   $BS(\is B'S)$   &7\\ 
$AB'(\is A'B)$  &1      &   $CC'$           &1,8\\ 
$AC'(\is A'C)$  &1,6    &   $CD'(\is C'D)$  &1,8\\ 
$AD'(\is A'D)$  &6      &   $CQ(\is C'Q)$   &7\\ 
$AQ(\is A'Q)$   &5      &   $CS(\is C'S)$   &4\\ 
$AR(\is A'R)$   &3,5,6  &   $CT(\is C'T)$   &4,8\\ 
$AS(\is A'S)$   &3      &   $DP(\is D'P)$   &2\\ 
$AT(\is A'T)$   &3      &   $DQ(\is D'Q)$   &2,5,7\\
$BB'$           &1      &   $DR(\is D'R)$   &5\\
$BC'(\is B'C)$  &1,4,6,7,8& $DT(\is D'T)$   &8\\ 
$BD'(\is B'D)$  &1,6,7,8&   $PS$            &2,4\\
$BP(\is B'P)$   &4      &   $PT$            &4\\
$BR(\is B'R)$   &6      &   $QS$            &2,7\\
\hline
\end{tabular}\end{center}







Therefore, adding any edge will make the graph contain a $K^6$-minor; hence adding any edge will make the graph intrinsically linked. So $G$ is maximally linkless.\end{proof}

Now that we have constructed $G$, we will use it to construct a class of maximally linkless graphs with asymptotically few edges.

\begin{thm} There exist maximally linkless graphs with $n$ vertices and $m\le \frac{14}{5}n$ edges for arbitrarily high $n$. \end{thm}

\begin{proof} Consider two flat embedded graphs $G_1,G_2$ which each contain triangles. Since there exists a disk with these triangles as its boundary in each of $G_1, G_2$, we can topologically move the edges so that the triangle is very small, then pull it out so that the triangle lies on a plane and the rest of the graph lies entirely on one side. Then we can form a combined graph by letting the triangles coincide, with $G_1$ entirely above the plane and $G_2$ entirely below (as in~\cite{KKM}). Note that if a cycle does not lie entirely within either $G_1$ or $G_2$, then it must pass through the triangle at least twice. Since the triangle has only three points, in any pair of cycles, one lies entirely within one of the original graphs. Hence if there exist linked cycles in the combined graph, there exist linked cycles within one of the original embeddings, a contradiction. Hence the graph formed by combining $G_1$ and $G_2$ upon a common triangle is again linkless.

Let $H$ be the graph formed by combining $k$ copies of $G$ along a common triangle. By the above, $H$ is also linkless. Suppose that this graph were not maximally linkless, and that we could add some edge $uv$ to it. Since each copy of $H$ is maximally linkless, $uv$ cannot be entirely within one copy of $H$, so each of $u,v$ are in different copies of $G$. Since $u$ is not adjacent to all the vertices of the common triangle, we could contract the copy of $G$ containing $v$ down to a single point, then contract the edge $uv$. This would create a copy of $G$ with an additional edge, which is impossible; hence $H$ is also maximally linkless.

Observe that $H$ has $n = 3 + 10k$ vertices and $m=3+28k$ edges. Hence for any value of $n\ge13$ which is $3\mod 10$, we have constructed a maximally linkless graph with $m=3+28(\frac{n-3}{10})= \frac{14}{5}n-\frac{27}{5}$. \end{proof}

Finally, we shall prove a lower bound on the on the number of edges in a maximally linkless graph.

\begin{thm} Let $G$ be a maximally linkless graph with $n$ vertices and $m$ edges. Then $m\ge 2n$.
\end{thm}

\begin{proof} Suppose for the sake of contradiction that there exist maximally linkless graphs with $m<2n$ edges, and let $H$ be such a graph with minimum number of vertices. Since $2m<4n$, the average degree is less than~$4$, so $H$ contains a vertex $v$ of degree at most~$3$. If $\deg(v)=1$, then $H$ cannot be maximally linkless, as we can connect $v$ to any neighbor of $u$, if $u$ is the original neighbor of $v$. If $\deg(v)=2$, then the two neighbors of $v$ must be connected, as we can draw an edge between them closely following the path of the edges to $v$, so $H-v$ is maximally linkless and has $m-2<2(n-1)$ edges, contradicting the minimality of $H$. If $\deg(v)=3$, then since $Y-\Delta$ transforms preserve linkless graphs, it follows that the three neighbors of $v$ must be connected, so $H-v$ has $m-3<n(2-1)$ edges, again contradicting the minimality of $H$. Hence we have a contradiction for all possible values of $\deg(v)$, so no such graph $H$ exists.
\end{proof}

\section{Conclusion}
We have shown that there exist maximally linkless graphs with $m\le \frac{14}{5}n$ edges, and that for all maximally linkless graphs, $m\ge 2n$. These two results show that the smallest possible asymptotic ratio of the number of edges to the number of vertices in a maximally linkless graph is between $2$ and $\frac{14}{5}=2.8$. It is unknown where the true constant lies within this range.

\section*{Acknowledgements}
The author gratefully acknowledges the support of Carnegie Mellon through their Summer Undergraduate Research Fellowship. In addition, he would like to thank Florian Frick for the excellent advice given throughout the summer research project.

\end{document}